%% file: unique-invariant.measures.tex
\documentclass[12pt,authoryear,letterpaper,reqno,oneside]{amsart}
\pdfoutput=1
\usepackage{microtype}
\usepackage{amsmath,amsfonts,amssymb}%
\usepackage[usenames]{color}
\usepackage{setspace}
\usepackage{comment}
\usepackage[colorlinks,citecolor=darkblue,urlcolor=darkblue,linkcolor=darkgreen]{hyperref}
\usepackage{stmaryrd}
\definecolor{darkred}{rgb}{0.5,0,0}
\definecolor{darkgreen}{rgb}{0, 0.3,0}
\definecolor{darkblue}{rgb}{0,0,0.6}
\definecolor{LightGray}{rgb}{.6,.6,.6}
\include{macros}

\usepackage[left=1.45in,right=1.45in,bottom=1.40in,top=1.40in]{geometry}

\begin{document}


\title[Orbits admitting a unique invariant measure]
{A classification of orbits admitting \\ a unique invariant measure}

\author[Ackerman]{Nathanael Ackerman}
\address{
\newline
Department of Mathematics\newline
Harvard University \newline
Cambridge, MA 02138 \newline
USA}
\email{nate@math.harvard.edu}

\author[Freer]{Cameron Freer}
\address{
\newline
Department of Brain and Cognitive Sciences\newline
Massachusetts Institute of Technology\newline
Cambridge, MA 02139 \newline
USA
}
\email{freer@mit.edu}

\author[Kwiatkowska]{\\Aleksandra Kwiatkowska}
\address{{\newline Mathematical Institute \newline
University of Bonn \newline
Endenicher Allee 60 \newline
53115 Bonn \newline
Germany}}
\email{akwiatk@math.uni-bonn.de}

\author[Patel]{Rehana Patel}
\address{\newline
Franklin W.\ Olin College of Engineering \newline
Needham, MA 02492 \newline
USA}
\email{rehana.patel@olin.edu}


\begin{abstract}
	We consider the 
space of countable structures with fixed underlying set
in a given countable language.
We show that the
number of ergodic 
probability measures on this space that are
$\sym$-invariant and
concentrated on 
a single isomorphism class
must be 
zero, or one, or continuum.
Further, 
such an
isomorphism class 
admits a unique $\sym$-invariant probability measure
precisely when
the structure
is 
highly homogeneous;
by
a result of Peter~J.\ Cameron, 
these are the structures that are
interdefinable with one
of the five reducts of the rational linear order $(\Rationals, <)$.
\end{abstract}

{\let\thefootnote\relax\footnotetext{
\indent \ 2010~{\it Mathematics Subject Classification:} \ 
03C98,
37L40,
60G09,
20B27.
%
}

\maketitle


\setcounter{page}{1}
\thispagestyle{empty}

\begin{small}
\begin{tiny}
\renewcommand\contentsname{\!\!\!\!}
\setcounter{tocdepth}{3}
\tableofcontents
\end{tiny}
\end{small}

\section{Introduction}
\label{intro-section}

A countable structure in a countable language can be said to admit a random symmetric construction when there is a probability measure on its isomorphism class (of structures having a fixed underlying set)
that is invariant under the logic action of $\sym$.
Ackerman, Freer, and Patel \cite{AFP} characterized those structures admitting such invariant measures. 
In this paper, we further explore this setting by determining
the possible numbers of such ergodic invariant measures, and by characterizing
when there is a unique invariant measure.

A dynamical system is said to be uniquely ergodic when it admits a unique, hence necessarily ergodic, invariant measure. 
In most classical ergodic-theoretic settings, the dynamical system consists of a measure space along with a single map, or at most a countable semigroup of transformations; unique ergodicity has been of longstanding interest for such systems.
In contrast, 
unique ergodicity for systems consisting of a larger space of transformations (such as the automorphism group of a structure) has been a focus of more recent research, 
notably
that of 
Glasner and Weiss \cite{MR1937832}, and of
Angel, Kechris, and Lyons \cite{MR3274785}.

When studying
continuous dynamical systems, 
one often 
considers
minimal flows, i.e., continuous actions on compact Hausdorff spaces such that each orbit is dense;
\cite{MR3274785} 
examines
unique ergodicity in this setting.
In the present paper, we are interested in unique ergodicity of actions where the underlying space need not be compact and there is just one orbit: We
characterize 
when the logic action of the group $\sym$ on an orbit
is uniquely ergodic.

Any transitive $\sym$-space
is isomorphic to the action of $\sym$ on the isomorphism class of a countable structure, restricted to
a fixed underlying set.
The main result of
\cite{AFP} 
states that 
such an isomorphism class 
admits at least one $\sym$-invariant measure
precisely when the structure has trivial definable closure.
Here we 
characterize those countable structures whose isomorphism classes admit \emph{exactly one} such measure, and show via a result of 
Peter J.\ Cameron
that the five reducts of $(\Rationals, <)$ are essentially the only ones.
Furthermore, if the isomorphism class of a countable structure
admits more than one $\sym$-invariant measure, it must
admit continuum-many ergodic such measures.

\subsection{Motivation and main results}
\label{background-subsection}
In this paper we consider,
for a given countable language $L$, 
the collection 
of countable $L$-structures 
having the 
natural numbers $\Nats$ as underlying set.
This collection can be made into a measurable space, denoted $\Str_L$, in a standard way, as we describe in Section~\ref{prelim-sec}.

The group $\sym$ of permutations of $\Nats$ acts naturally on $\Str_L$ by permuting the underlying set of elements. This action is known as the
\emph{logic action} of $\sym$ on $\Str_L$, and has been studied extensively in descriptive set theory.
For details, see \cite[\S2.5]{MR1425877} or \cite[\S11.3]{MR2455198}.
Observe
that the $\sym$-orbits of $\Str_L$ are precisely the isomorphism classes of structures in $\Str_L$.

By an \emph{invariant measure} on $\Str_L$, we will always mean a
Borel probability measure on $\Str_L$ that is invariant under the logic action of $\sym$.
We are specifically interested in those invariant measures on $\Str_L$ that assign measure $1$ to a single orbit, i.e., the isomorphism class in $\Str_L$ of some countable $L$-structure $\M$. 
In this case we say that 
the orbit of $\M$ \emph{admits} an invariant measure, or simply that
$\M$ admits an invariant measure. 

When a countable structure $\M$ admits an invariant measure, this measure can be thought of as providing a symmetric random construction of $\M$.
The main result of
\cite{AFP} describes precisely when such a construction is possible:
A structure $\M\in\Str_L$ admits an invariant measure if and only if 
definable closure 
in
$\M$ is trivial, i.e., the pointwise stabilizer in $\Aut(\M)$ of any finite tuple 
fixes no additional elements. But even when there \emph{are} invariant measures concentrated on the orbit of $\M$, it is not obvious 
how many different ones there are.

If
an orbit admits at least two invariant measures, there are trivially always continuum-many such measures, because a convex combination of any two gives an invariant measure on that orbit, and these combinations yield distinct measures.  
It is therefore useful to count instead the invariant measures that are not decomposable in this way, namely the \emph{ergodic} ones.
It is a standard fact that the invariant measures on $\Str_L$ form a simplex in which the ergodic invariant measures are precisely the \emph{extreme} points, i.e., those that cannot be written as a nontrivial convex combination of invariant measures. Moreover, every invariant measure is a mixture of these extreme invariant measures. (For more details, see \cite[Lemma~A1.2 and Theorem~A1.3]{MR2161313} and \cite[Chapters 10 and 12]{MR1835574}.)
Thus when counting invariant measures on an orbit, the interesting quantity to consider is the number of ergodic invariant measures.

Many natural examples admit more than one invariant measure. For instance, consider the \ER\ 
\cite{MR0120167} 
construction 
$G(\Nats, p)$ 
of the Rado graph, a countably infinite random graph in which edges have
independent probability $p$, where $0<p<1$. 
This yields continuum-many ergodic invariant measures concentrated on the orbit of the Rado graph, as
each value of $p$ leads to a different ergodic invariant measure.

On the other hand, some countable structures admit just one invariant measure.
One such example is well-known:
The 
rational linear order $(\Rationals, <)$ admits 
a unique 
invariant measure, which can be described as follows.
For every 
finite
$n$-tuple of distinct elements of $\Rationals$,
each of the $n!$-many orderings of the $n$-tuple must be assigned the same
probability, by $\sym$-invariance.
This collection
of finite-dimensional marginal distributions determines a (necessarily unique)
invariant 
measure on $\Str_L$, by the Kolmogorov extension theorem (see, e.g., \cite[Theorems~6.14 and 6.16]{MR1876169}). 
This probability measure 
can be shown to be
concentrated on the orbit of $(\Rationals, <)$ and
is sometimes known as the \emph{Glasner--Weiss measure};
for details, see 
\cite[Theorems 8.1 and~8.2]{MR1937832}.
We will discuss this measure and its construction further in Section~\ref{uniqueness-for-hh-section}.

In fact, these examples illustrate the only possibilities:
Either a countable structure admits
no invariant measure, or a unique invariant measure,
or continu\-um-many 
ergodic invariant measures.
Furthermore, 
a countable structure admits a unique invariant measure precisely when it has the property known as \emph{high homogeneity}.
The main result of this paper is the following trichotomy. It refines the dichotomy obtained in \cite{AFP}.

\begin{theorem}\label{simpler-theorem}
Let $\M$ be a countable structure in a countable language $L$.
Then exactly one of the following holds:

\begin{itemize}
\item[($0$)] The structure $\M$ has nontrivial definable closure, in which case there is
no $\sym$-invariant Borel probability measure on $\Str_L$ that is concentrated on 
the orbit of $\M$.

\item[($1$)] The structure $\M$ is 
highly homogeneous, in which case there is a 
unique $\sym$-invariant Borel probability measure 
on $\Str_L$ that is concentrated on 
the orbit of $\M$.

\item[(${2^{\aleph_0}}$)] There are continuum-many ergodic $\sym$-invariant
Borel probability measures 
on $\Str_L$ that are concentrated on 
the orbit of $\M$.
\end{itemize}
\end{theorem}

Moreover, 
by a result of Peter~J.\ Cameron,
the case where 
$\M$ is 
highly homogeneous 
is equivalent to
$\M$ 
being
interdefinable with a definable reduct (henceforth \emph{reduct}) of $(\mathbb{Q}, <)$, of which there are five.
In particular, this shows the known 
invariant measures on
these five to be canonical.

\subsection{Additional motivation}

The present work has been motivated by 
further considerations, which we now describe.

Fouch\'e and Nies
\cite[\S15]{LogicBlog11} 
describe one notion
of an algorithmically random presentation of a given countable structure; see also \cite{MR3016251}, \cite{MR3023193}, and \cite{MR1649077}.
In the case of
the rational linear order $(\Rationals, <)$, 
they note that their notion of randomness is in a sense canonical by virtue of the unique ergodicity of the orbit of $(\Rationals, <)$.
Hence one may
ask which other orbits of countable structures are uniquely ergodic.
Theorem~\ref{simpler-theorem}, along with the result of Cameron, shows that the orbits of
$(\Rationals, <)$ and of its reducts are essentially the only instances.

We also
note a connection with ``Kolmogorov's example'' of a transitive but non-ergodic action of $\sym$, described by Vershik in
\cite{MR2015459}.
Many 
settings in 
classical
ergodic theory permit at most one invariant measure. For example, when a separable locally compact group $G$ acts continuously and transitively on a Polish space $X$, there is at most one $G$-invariant probability measure on $X$ \cite[Theorem~2]{MR2015459}.
The action of $\sym$, however,
allows for continuum-many ergodic invariant measures on the same orbit, as 
noted above 
in the case of the \ER\ constructions; for more details, see \cite[\S3]{MR2015459}.
Indeed, many specific orbits with this property are known,
but 
the present work strengthens the sense in which this is 
typical
for $\sym$ and uniqueness is rare: There are essentially merely five exceptions to the rule of having either continuum-many ergodic invariant measures or none.

It may be an interesting question to further understand the structure of the simplex of invariant measures in the non-uniquely-ergodic case --- when, as we show in the case of a single orbit, there are continuum-many ergodic such measures or none.
Note, however, that this space will often not be compact, as the actions we consider are on spaces that are usually not compact.

\section{Preliminaries}
\label{prelim-sec}

In this paper, $L$ will always be a countable language.
We consider the space $\Str_L$ of countable $L$-structures 
having
underlying set $\Nats$, equipped with the $\sigma$-algebra of Borel sets generated by the topology described in Definition~\ref{widehat-def} below.
We will often use the notation $\x$ to denote the finite tuple of variables
$x_0\cdots x_{n-1}$, where $n = |\x|$. 

Recall that $\Lwow(L)$ denotes the infinitary language based on $L$ 
consisting of formulas that can have
countably infinite conjunctions and disjunctions, 
but only
finitely many quantifiers and free variables; for details, see 
\cite[\S16.C]{MR1321597}.

\begin{definition}
\label{widehat-def}
Given a formula 
$\varphi \in \Lwow(L)$
and
$n_0, \ldots, n_{j-1} \in \Nats$,
where $j$ is the number of free variables of $\varphi$,
define
\[
\llrr{\varphi(n_0, \dots, n_{j-1})} \defas
 \{\M \in \Str_L \st \M\models \varphi(n_0, \dots, n_{j-1})\}.
\]
\end{definition}
Sets of this form are closed under finite intersection, and 
form a basis  for the topology of $\Str_L$.

Consider
$\sym$,
the permutation group of the natural numbers $\Nats$.
This group acts on $\Str_L$ via the \defn{logic action}:
For $g\in\sym$ and $\M$, $\N\in\Str_L$, we define
$g \cdot \M = \N$ to hold
whenever
\[
R^\N(s_0, \ldots, s_{k-1}) 
\quad \text{if and only if} \quad
R^\M\bigl(g^{-1}(s_0), \ldots,  g^{-1}(s_{k-1}) \bigr)
\]
for all relation symbols $R\in L$ and $s_0, \ldots, s_{k-1} \in\Nats$, where $k$ is the arity of $R$, and similarly with constant and function symbols.
Observe that the orbit of a structure under the logic action is its isomorphism class in $\Str_L$; every such orbit is Borel by Scott's isomorphism theorem.
For more details on the logic action, see \cite[\S16.C]{MR1321597}.

We define an
\defn{invariant measure}
on $\Str_L$
to be
a Borel probability measure $\mu$
on $\Str_L$ that is invariant under the logic action of $\sym$ on $\Str_L$,
i.e., $\mu(X) = \mu(g\cdot X)$ for every Borel set $X \subseteq \Str_L$ and $g \in \sym$.
When an invariant measure on $\Str_L$ is concentrated on the orbit of some structure in $\Str_L$,
we may restrict
attention to this orbit, 
and speak
equivalently
of an 
invariant measure on 
the orbit itself.

In this paper, we are interested in invariant measures that are \emph{ergodic}.
Given an action of a group $G$ on a set $X$, and an element $g\in G$, we write $g x$ to denote the image of $x\in X$ under $g$, and for $A\subseteq X$ we write $g A \defas \{ g x \st x \in A\}$.

\begin{definition}
\label{ergodic}
Consider a Borel action of a Polish group $G$ on a standard Borel space $X$.  A probability measure $\mu$ on $X$ is \defn{ergodic} when for every Borel $B \subseteq X$ satisfying $\mu(B \vartriangle g^{-1}B ) = 0$ for each $g\in G$, either $\mu(B) = 0$ or $\mu(B)= 1$.
\end{definition}
In other words, an ergodic measure (with respect to a particular action of $G$) is one that assigns every almost $G$-invariant set either zero or full measure.
In our setting, $X$ will be one of $\Reals^\omega$ or $\Str_L$, and we will consider the group action of $\sym$ on $\Reals^\omega$ that permutes coordinates, and the logic action of $\sym$ on $\Str_L$.

A structure $\M\in\Str_L$ has \emph{trivial definable closure} when the pointwise stabilizer in $\Aut(\M)$ of an arbitrary finite tuple of $\M$ fixes no additional points:

\begin{definition}
Let $\M\in\Str_L$.
For a tuple $\a \in\M$,  the \defn{definable closure} of $\a$ in $\M$, written $\dcl_\M(\a)$, is the set of 
elements of
$\M$ that are fixed by every automorphism of $\M$
that fixes $\a$ pointwise.
The structure $\M$ has \defn{trivial definable closure} when $\dcl_\M(\a) = \a$ for every (finite) tuple $\a\in\M$.
\end{definition}

The easier direction of the main theorem of \cite{AFP} states that any structure admitting an invariant measure must have trivial definable closure.

\begin{theorem}[{\cite[Theorem~4.1]{AFP}}]
\label{negative-AFP}
Let
$\M\in\Str_L$. If $\M$ does not have trivial definable closure, then 
$\M$ does not admit an invariant measure.
\end{theorem}

This corresponds to case (0) of Theorem~\ref{simpler-theorem}.

\subsection{Canonical structures and interdefinability}
\label{canonical-prelim}

In the proof of our main theorem, we will work in the setting of 
\emph{canonical languages} and \emph{canonical structures}.
We provide a brief description of these notions here; for
more details, see \cite[\S2.5]{AFP}.

\begin{definition}
Let $G$ be a closed 
subgroup of $\sym$, and consider the action of $G$ on $\Nats$.
Define the \defn{canonical language for $G$} to be 
the (countable) relational language $L_{G}$ that
consists of, for
each $k\in\Nats$ and
$G$-orbit
$E\subseteq \Nats^k$,
a $k$-ary relation symbol $R_E$.
Then define the \defn{canonical structure for $G$}
to be the structure
$\Gbar \in \Str_{L_{G}}$ in which,
for each
$G$-orbit $E$,
the interpretation
of $R_E$ is the set $E$.
\end{definition}

\begin{definition}
\label{def-canonical}
Given a structure $\M \in \Str_L$, 
define the \defn{canonical language for $\M$},
written $L_{\Mbar}$,  to be the countable relational language $L_{G}$ where $G \defas \Aut(\M)$. Similarly, 
define the \defn{canonical structure for $\M$},
written $\Mbar$,  to be the countable $L_{\Mbar}$-structure $\Gbar$.
\end{definition}

Structures that are \emph{interdefinable}, in the following sense,
will be regarded as interchangeable for purposes of our classification.

\begin{definition}
\label{def-interdef}
Let $\M$ and $\N$ be 
structures in (possibly different) countable languages, both having
underlying set $\Nats$.
Then $\M$ and $\N$ are said to be \defn{interdefinable} 
when they have the same canonical language and same canonical structure.
\end{definition}
Note that
two structures are interdefinable if and only 
if
there is an
\emph{$\Lwow$-interdefinition} between them,
in the terminology of \cite[Definition~2.11]{AFP};
see the discussion after \cite[Lemma~2.13]{AFP} for details.

By Definitions~\ref{def-canonical} and \ref{def-interdef}, it is immediate that a structure $\M\in\Str_L$ and its canonical structure $\Mbar$ are interdefinable.

\begin{proposition}
\label{wlog-prop}
Let $\M\in\Str_L$. There is a Borel bijection, respecting the action of $\sym$, between the orbit of $\M$ 
in $\Str_L$
and the orbit of its canonical structure $\Mbar$
in $\Str_{L_{\Mbar}}$.
In particular, this map induces a bijection between
the set of ergodic invariant measures 
on the orbit of $\M$ 
and the set of ergodic invariant measures
on the orbit of $\Mbar$.
\end{proposition}
\begin{proof}
First observe that the orbit of $\M$ and the orbit of $\Mbar$ are each Borel spaces that inherit the logic action.
The structures $\M$ and $\Mbar$ are interdefinable,
and so
\cite[Lemma~2.14]{AFP}  applies. The
proof of this lemma
provides explicit maps between $\Str_L$ and $\Str_{L_{\Mbar}}$
which, when restricted respectively to the orbit of $\M$ and of $\Mbar$,
have
the desired property.
\end{proof}

We immediately obtain the following corollary.
\begin{corollary}
\label{wlog-corollary}
Let $\M\in\Str_L$. Then $\M$ and its canonical structure $\Mbar$ admit the same number of ergodic invariant measures.
\end{corollary}

The following two results are straightforward.

\begin{lemma}[{\cite[Lemma~2.15]{AFP}}]
\label{trivialdcl-interdef}
Let $\M$ and $\N$ be interdefinable structures in (possibly different) countable languages, both having underlying set $\Nats$. Then $\M$ has trivial definable closure if and only if $\N$ does.
\end{lemma}

\begin{corollary}
\label{trivialdcl-canonical}
Let $\M\in\Str_L$. Then $\M$ has trivial definable closure if and only if its canonical structure $\Mbar$ has trivial definable closure.
\end{corollary}

For any $\M\in\Str_L$,
we will see 
in Lemma~\ref{wlog-lemma}
that $\M$ is highly homogeneous if and only if $\Mbar$ is;
combining
this fact
with Corollaries~\ref{wlog-corollary} and \ref{trivialdcl-canonical},
when proving
Theorem~\ref{simpler-theorem}
it will suffice to
consider instead
the number of ergodic invariant measures admitted by the canonical structure $\Mbar$.

\subsection{Ultrahomogeneous structures}
\label{subsec-ultrahom}
Countable \emph{ultrahomogeneous} relational structures play an important role throughout this paper,
as canonical structures are ultrahomogeneous and canonical languages are relational.

\begin{definition}
A countable structure $\M$ is \defn{ultrahomogeneous} if every isomorphism between finitely generated substructures of $\M$ can be extended to an automorphism of $\M$.
\end{definition}

The following fact is folklore (see also the discussion following \cite[Proposition~2.17]{AFP}).

\begin{proposition}
\label{canonical-is-ultrahom}
Let $\M\in\Str_L$. The canonical structure $\Mbar$ is ultrahomogeneous.
\end{proposition}

Ultrahomogeneous structures can be given 
particularly 
convenient
$\Lwow(L)$
axiomatizations via
\emph{pithy 
$\Pi_2$} sentences, which
can be thought of as ``one-point
extension axioms''.

\begin{definition}[{\cite[Definitions~2.3 and 2.4]{AFP}}]
A sentence in $\Lwow(L)$ is $\Pi_2$ when it is of the form $(\forall
\x)(\exists \y)\psi(\x, \y)$, where the (possibly empty) tuple $\x\y$
consists of distinct variables, and $\psi(\x,\y)$ is quantifier-free.  A
countable theory $T$ of $\Lwow(L)$ is $\Pi_2$ when every sentence
in $T$
is $\Pi_2$.

A $\Pi_2$ sentence $(\forall\x)(\exists\y)\psi(\x,\y) \in \Lwow(L)$,
where $\psi(\x,\y)$ is quanti\-fier-free, is said to be \defn{pithy} when
the tuple $\y$ consists of precisely one variable.  A countable $\Pi_2$
theory $T$ of $\Lwow(L)$ is said to be pithy when every sentence in $T$ is
pithy.  Note that we allow the degenerate case where $\x$ is the empty
tuple and $\psi$ is of the form $(\exists y)\psi(y)$.
\end{definition}

The following result 
follows from 
essentially the same
proof 
as 
\cite[Proposition~2.17]{AFP}.
It states that a Scott sentence for an ultrahomogeneous relational structure is equivalent to a theory of a particular syntactic form.
\begin{proposition}
\label{pithyconseq}
Let $L$ be relational and let $\M\in\Str_L$ be ultrahomogeneous.
There is a countable
$\Lwow(L)$-theory,
every sentence of which is
pithy $\Pi_2$,
and all of whose countable models are
isomorphic to 
$\M$.
\end{proposition}
We will call this theory the \defn{\Fr\ theory of $\M$}.

\subsection{Highly homogeneous structures}
\label{subsec-highly-hom}

\emph{High homogeneity} is the key notion in our characterization of structures admitting a unique invariant measure.

\begin{definition}[{\cite[\S2.1]{MR1066691}}]
A
structure $\M\in\Str_L$
is \defn{highly homogeneous} 
when,
for each $k\in\Nats$ and for every pair of $k$-element sets $X,Y\subseteq \M$, there is some $f\in \Aut(\M)$ such that
$Y = \{ f(x) \st x \in X\}$. 
\end{definition}

The following lemma is immediate, and allows us to generalize the notion of high homogeneity to permutation groups.

\begin{lemma}
\label{wlog-lemma}
A structure $\M\in\Str_L$
is highly homogeneous if and only if its canonical structure is.
\end{lemma}

\begin{definition}[{\cite[\S2.1]{MR1066691}}]
A closed permutation group $G$ on $\Nats$ 
is called \defn{highly homogeneous} 
when
its canonical structure $\Gbar$ is highly homogeneous.
\end{definition}

The crucial fact about highly homogeneous structures 
is the following.

\begin{lemma}
\label{ultra-highhom}
Let $L$ be relational and let $\M\in\Str_L$ be 
ultrahomogeneous. 
Then $\Aut(\M)$ is highly homogeneous  
if and only if for any 
$k\in \Nats$, all $k$-element 
substructures of $\M$
are isomorphic.
\end{lemma}
\begin{proof}
Let $\M\in \Str_L$ be ultrahomogeneous,
and 
let $X$ and $Y$ be arbitrary substructures of $\M$ of size $k$.
If $\Aut(\M)$ is highly homogeneous, 
then 
$X$ and $Y$ are isomorphic 
via the restriction of
any $f \in \Aut(\M)$ such that 
$Y = \{ f(x) \st x \in X\}$. 
Conversely, 
if
there is some isomorphism of structures $g\colon X\to Y$, then by the ultrahomogeneity of $\M$, there is some $f\in\Aut(\M)$ extending $g$ 
to all of $\M$.
\end{proof}

Highly homogeneous structures have been classified explicitly by Cameron \cite{MR0401885},
and characterized (up to interdefinability) as the five reducts of $(\Rationals, <)$,
as we now describe.

Let $(\Rationals, <)$ be the set of rational numbers equipped with the usual order.
The following three relations are definable within $(\mathbb{Q}, <)$:

\noindent (1) The ternary linear \emph{betweenness} relation $B$, given by
\[ B(a, b, c) \iff (a < b < c) \vee (c < b < a).\]
(2) The ternary \emph{circular order} relation $K$, given by
\[K(a, b, c)\iff (a < b < c) \vee (b < c < a) \vee (c < a < b).\]
(3) The quaternary \emph{separation} relation $S$, given by
\begin{eqnarray*}
S(a, b, c, d) &\iff& \bigl(K(a, b, c) \wedge K(b, c, d) \wedge K(c, d, a)\bigr)
\\
&&
\ \ \ \vee \ 
\bigl(K(d, c, b) \wedge K(c, b, a) \wedge K(b, a, d)\bigr).
\end{eqnarray*}
The structure $(\Rationals, B)$ can be thought of as forgetting the direction of the order, $(\Rationals, K)$ as gluing the rational line into a circle, and $(\Rationals, S)$ as forgetting which way is clockwise on this circle.

The following is a consequence of Theorem 6.1 in Cameron \cite{MR0401885};
see also (3.10) of 
\cite[\S3.4]{MR1066691}.
For 
further
details, see 
\cite[Theorem~6.2.1]{MR2800979}.

\begin{theorem}[Cameron]
\label{highhom-perm-thm}
Let $G$ be a highly homogeneous structure.
 Then $G$ 
is interdefinable with 
one of the following: the set $\mathbb{Q}$ (in the empty language),
$(\mathbb{Q},<)$, $(\mathbb{Q},B)$, $(\mathbb{Q},K)$, or $(\mathbb{Q},S)$.
\end{theorem}

Notice that these five structures all have trivial definable closure;
this will imply, in Lemma~\ref{highly-homog-meas-from-struct},
that the 
orbit
of each highly homogeneous structure 
admits a unique invariant measure.

\subsection{Borel $L$-structures and ergodic invariant measures}
\label{BorelL}

Aldous, Hoover, and Kallenberg have characterized ergodic invariant measures on $\Str_L$ 
in terms of a certain sampling procedure involving continuum-sized objects;
for 
details, see \cite{MR2426176} and \cite{MR2161313}.

We will obtain ergodic invariant measures via a special case of this procedure, by sampling from a particular kind of continuum-sized structure, called a \emph{Borel $L$-structure}.
For more on the connection between Borel $L$-structures and the Aldous--Hoover--Kallenberg representation, see \cite[\S6.1]{AFP}.

\begin{definition}[{\cite[Definition~3.1]{AFP}}]
\label{BorelLstructure}
Let $L$ be relational, and
let $\PP$ be an $L$-structure whose underlying set is the set $\Reals$ of real numbers.  We say
that $\PP$ is a \defn{Borel $L$-structure} if for all relation symbols $R \in L$,
the set
\[
\{\a \in \PP^j \st \PP\models R(\a)\}
\]
is a Borel subset of $\Reals^j$, where $j$ is the arity of $R$.
\end{definition}

The sampling procedure is given by
the following map $\FP$
that takes
each sequence of elements of $\PP$ to the corresponding structure with underlying 
set
$\Nats$.

\begin{definition}[{\cite[Definition~3.2]{AFP}}]
Let $L$ be relational and let $\PP$ be a Borel $L$-structure.
The map $\FP\colon \Reals^\omega \to \Str_L$ is defined as follows.
For $\t = (t_0, t_1, \ldots)\in\Reals^\omega$, let $\FP(\t)$ be the
$L$-structure with underlying set $\Nats$ satisfying
\[
\FP(\t) \, \models \,R(n_1, \dots, n_j) \quad \Leftrightarrow \quad 
\PP \models R(t_{n_1}, \dots, t_{n_j})
\]
for all $n_1, \dots, n_j \in \Naturals$
and for every relation symbol $R \in L$, and for which equality is inherited 
from $\Naturals$, i.e.,
\[
\FP(\t)
\, \models \,(m \neq  n)
\]
if and only if
$m$ and $n$ are distinct natural numbers.
\end{definition}
The map $\FP$ is Borel measurable \cite[Lemma~3.3]{AFP}.
Furthermore, $\FP$
is 
an $\sym$-map,
i.e.,
$\sigma \FP(\t)=\FP (\sigma \t)$
for every $\sigma\in\sym$ and $\t\in\Reals^\omega$.

The pushforward of $\FP$ gives rise to 
an ergodic
invariant measure, as we will see in Proposition~\ref{ergodic-measure}.

\begin{definition}[{\cite[Definition~3.4]{AFP}}]
\label{muPm-def}
Let $L$ be relational, let $\PP$ be a Borel $L$-structure,
and let $m$ be a probability measure on $\Reals$.
Define the measure $\mu_{(\PP,m)}$ on $\Str_L$ to be
\[
\mu_{(\PP, m)} \defas
 m^{\infty}\circ \FP^{-1}.
\]

Note that $m^\infty\bigl(\FP^{-1}(\Str_L)\bigr) =1$, and so
$\mu_{(\PP,m)}$ is a probability measure, namely the distribution of a random
element in $\Str_L$
induced via $\FP$ by an $m$-i.i.d.\ sequence on $\Reals$.
\end{definition}

By \cite[{Lemma~3.5}]{AFP},
$\mu_{(\PP, m)}$ is an invariant
measure on $\Str_L$.
In fact, $\mu_{(\PP, m)}$ is ergodic:
Aldous showed that 
it
is ergodic
for finite relational languages
\cite[Lemma~7.35]{MR2161313}; 
we
require the following extension
of this result to the setting of countable (possibly infinite) relational languages,
whose proof we include here for completeness.

\begin{proposition}
\label{ergodic-measure}
Let $L$ be relational, let
$\PP$ be a Borel $L$-structure, and let $m$ be
a 
probability measure on $\Reals$. 
Then the measure $\mu_{(\PP,m)}$ is ergodic. 
\end{proposition}
\begin{proof}
First note that the measure $m^\infty$ on $\Reals^\omega$ is ergodic by the Hewitt--Savage 0--1 law; for details, see \cite[Corollary~1.6]{MR2161313} and \cite[Theorem~3.15]{MR1876169}.

Write $\mu\defas \mu_{(\PP,m)}$.
Let $B\subseteq \Str_L$ be Borel and suppose that 
$\mu(B\vartriangle \sigma^{-1} B)=0$
for every $\sigma\in S_\infty$.
We will show that either $\mu(B)=0$ or $\mu(B)=1$.

Let $\t\in \Reals^\omega$ and $\sigma\in\sym$.
We have
\[\t\in \sigma^{-1} \FP^{-1}(B) \iff \FP(\sigma \t) \in B,\]
where $\sigma$ and $\sigma^{-1}$ act on $\Reals^\omega$,
and
\[\t\in \FP^{-1}( \sigma^{-1} B ) \iff \sigma \FP(\t)\in B,\]
where $\sigma$ and $\sigma^{-1}$ act on $\Str_L$ via the logic action.

Now, $\sigma \FP(\t)=\FP (\sigma \t)$, and so
 \[ \FP^{-1}(\sigma^{-1} B)=\sigma^{-1} \FP^{-1}(B) .\]
Using this fact, we have
\begin{eqnarray*}
0&=&\mu(B\vartriangle \sigma^{-1}B)\\
&=& m^\infty\bigl(\FP^{-1}(B\vartriangle \sigma^{-1} B)\bigr)\\
&=& m^\infty\bigl(\FP^{-1}(B)\vartriangle \FP^{-1}( \sigma^{-1} B)\bigr)\\
&=& m^\infty\bigl(\FP^{-1}(B)\vartriangle \sigma^{-1}\FP^{-1}(B)\bigr)\\
&=& m^\infty (A\vartriangle \sigma^{-1}A),
\end{eqnarray*}
where $A\defas \FP^{-1}(B)$.

Because $m^\infty$ is ergodic and 
$m^\infty(A\vartriangle \sigma^{-1} A)=0$
for every $\sigma\in S_\infty$, either
$m^\infty(A)=0 $ or $m^\infty(A)=1$ must hold. But then as $\mu(B)=m^\infty(A)$, 
either
$\mu(B)=0$ or $\mu(B)=1$, as desired.
\end{proof}

Not all ergodic invariant measures are of the form $\mu_{(\PP, m)}$: For example, it can be shown that the distribution of an \ER\ graph $G(\Nats, p)$ for $0<p<1$, each of
which is concentrated on the orbit of the Rado graph, is not of this form.
However,
Petrov and Vershik \cite{MR2724668} have shown that the orbit of the Rado graph admits an invariant measure of the form $\mu_{(\PP, m)}$ (in our terminology).
More generally,
the proof of \cite[Corollary~6.1]{AFP} shows that 
whenever
an orbit
admits an invariant measure, it admits one of 
the form $\mu_{(\PP, m)}$.
Note that this class of invariant measures also occurs elsewhere; see
Kallenberg's
 notion of \emph{simple arrays} \cite{MR1702867} and,
in the case of graphs, the notions of \emph{random-free graphons}
\cite[\S10]{MR3043217} or $0$--$1$ \emph{valued graphons} \cite{MR2815610}.

\subsection{Strong witnessing and the existence of invariant measures}

We now consider how to obtain ergodic invariant measures concentrated on a particular orbit.
We will do so by obtaining ergodic invariant measures concentrated on the class of models in $\Str_L$ of a particular \Fr\ theory $T$, 
where this class is the desired orbit.

A measure $m$ on $\Reals$ is said to be \defn{nondegenerate}
when every
nonempty open set has positive measure, and \defn{continuous} when it
assigns measure zero to every singleton.

\begin{definition}[{\cite[Definition~3.8]{AFP}}]
\label{Twitness}
Let $\PP$ be a Borel $L$-structure and let $m$ be a probability measure on
$\Reals$.  Suppose $T$ is a countable pithy $\Pi_2$ theory of $\Lwow(L)$.
We say that the pair $(\PP, m)$ \defn{witnesses $T$} if for every sentence
$(\forall \x)(\exists y)\psi(\x,y) \in T$, and for every tuple $\a \in
\PP$ such that $|\a| = |\x|$,  we have either
\begin{itemize}
\item[\emph{(i)}]
\quad $\PP \models \psi(\a,b)$ for some $b\in\a$, or
\vspace{5pt}
\item[\emph{(ii)}]
\quad $m\bigl(\{ b \in \PP \st \PP \models \psi(\a,b)\}\bigr ) > 0$.
\end{itemize}
We say that \defn{$\PP$ strongly witnesses $T$} when, for every
nondegenerate continuous probability measure $m$ on $\Reals$, the pair $(\PP,m)$
witnesses $T$.
\end{definition}

\begin{proposition}[{\cite[Theorem~3.10]{AFP}}]
\label{BorelLStructuresWitnessingTLeadToInvariantMeasures}
Let $L$ be relational,
let $T$ be a countable pithy $\Pi_2$ theory of $\Lwow(L)$, and let $\PP$ be
a Borel $L$-structure.  Suppose $m$ is a continuous probability measure on
$\Reals$ such that $(\PP,m)$ witnesses $T$.  Then $\mu_{(\PP,m)}$ is
concentrated on the set of structures in $\Str_L$ that are models of $T$.
\end{proposition}

The main theorem of \cite{AFP} states
that a countably infinite structure in a countable language admits at least one invariant measure if and only if it has trivial definable closure. 
The easier direction is stated in Theorem~\ref{negative-AFP} above. Proposition~\ref{AFP-Borel-L}, 
which we will need in the proof of Theorem~\ref{simpler-theorem},
is the key result used in Theorem~\ref{positive-AFP}, essentially the harder direction of \cite{AFP}.

\begin{proposition}[{\cite[Theorem~3.19 and Lemma~3.20]{AFP}}]
\label{AFP-Borel-L}
Let $L$ be relational and let $\M\in\Str_L$ be ultrahomogeneous.
If $\M$ 
has trivial definable closure, then there is a Borel $L$-structure $\PP$ that strongly witnesses the \Fr\ theory of $\M$.
\end{proposition}

\begin{theorem}[{\cite[Theorem~3.21]{AFP}}]
\label{positive-AFP}
Let $L$ be relational and let $\M\in\Str_L$ be ultrahomogeneous.
If $\M$ 
has trivial definable closure, then  $\M$ admits an invariant measure.
\end{theorem}
\begin{proof}
There is 
a Borel $L$-structure $\PP$ 
that strongly
witnesses the \Fr\ theory of $\M$,
by Proposition~\ref{AFP-Borel-L}.
Let $m$ be 
any nondegenerate  continuous probability measure $m$ on $\Reals$ (e.g., a Gaussian).
Then
by Proposition~\ref{BorelLStructuresWitnessingTLeadToInvariantMeasures},
the invariant measure $\mu_{(\PP, m)}$ is 
concentrated on 
the set of
models of the \Fr\ theory of $\M$ in $\Str_L$.
In particular, 
$\mu_{(\PP, m)}$ is 
concentrated on the
orbit
of $\M$.
\end{proof}

Finally, we establish a lemma about measures of the form $\mu_{(\PP, m)}$.
Recall the notation $\llrr{\varphi}$ from Definition~\ref{widehat-def}.



\begin{lemma}
\label{mu-vs-m-n}
Let $L$ be relational, let $\M\in\Str_L$ be ultrahomogeneous, and 
let
$T$ 
be the \Fr\ theory of $\M$.
Suppose that $\PP$ is a Borel $L$-structure
that strongly
witnesses $T$.
Let $m$  be a
nondegenerate 
continuous probability measure on $\Reals$. 
Then for every $n\in\Nats$  and  every $\Lwow(L)$-formula $\varphi$ having $n$ free variables,
\[
\mu_{(\PP, m)}
\bigl(
\llrr{\varphi(0, \ldots, n-1)}
\bigr)
 = m^n\bigl(\{ \a\in\Reals^n \st \PP \models \varphi(\a)\}\bigr).
\]
\end{lemma}
\begin{proof}
By \cite[{Lemma~3.6}]{AFP},
because $m$ is continuous, $\mu_{(\PP, m)}$ is 
concentrated on the isomorphism classes 
in $\Str_L$
of countably infinite substructures of $\PP$.

Because $\M$ is ultrahomogeneous, for every $\Lwow(L)$-formula $\varphi(\x)$ there is some quantifier-free $\psi(\x)$ such that 
\[
\M \models (\forall \x) \bigl ( \varphi(\x) \leftrightarrow \psi(\x) \bigr).
\]
Because $\PP$ strongly witnesses $T$, by \cite[{Lemma~3.9}]{AFP} we have that
$\PP \models T$. Hence
\[
\PP \models (\forall \x) \bigl ( \varphi(\x) \leftrightarrow \psi(\x) \bigr).
\]
In particular, if 
a sequence of reals
determines a substructure of $\PP$ that is isomorphic to $\M$, then this substructure is in fact $\Lwow(L)$-elementary.

Therefore, as $\PP$ strongly witnesses $T$, 
by Proposition~\ref{BorelLStructuresWitnessingTLeadToInvariantMeasures} and Definition~\ref{muPm-def}
the probability measure $m^\infty$ concentrates on sequences of reals that determine
elementary substructures of $\PP$.
Hence the probability that a structure sampled according to 
$\mu_{(\PP, m)}$
satisfies $\varphi(0, \ldots, n-1)$ is 
equal to
\[
m^n\bigl(\{ \a\in\Reals^n \st \PP \models \varphi(\a)\}\bigr),
\]
as desired.
\end{proof}

\section{The number of ergodic invariant measures}
\label{uniqueness-for-hh-section}

In this section we prove our main result, Theorem~\ref{simpler-theorem}.

\subsection{Unique invariant measures}
We now 
show that every ultrahomogeneous highly homogeneous structure admits a unique invariant measure.
Recall Cameron's result, Theorem~\ref{highhom-perm-thm}, that the highly homogeneous structures are (up to interdefinability) precisely the five reducts of the rational linear order $(\Rationals, <)$.

\begin{lemma}
\label{highly-homog-meas-from-struct}
Let $L$ be relational and let $\M\in\Str_L$ be ultrahomogeneous.
If $\M$ is
highly homogeneous, then there is an
invariant measure on the isomorphism class of $\M$ in 
$\Str_L$.
\end{lemma}
\begin{proof}
We can check directly that each reduct of $(\Rationals, <)$
has trivial definable closure.
By Theorem~\ref{highhom-perm-thm} and the hypothesis that $\M$ is highly homogeneous,
$\M$ is interdefinable with one of these five. 
Hence $\M$ also has trivial definable closure
by Lemma~\ref{trivialdcl-interdef}.
Therefore by Theorem~\ref{positive-AFP}, there is an invariant measure 
on the isomorphism class of $\M$ in $\Str_L$.
\end{proof}

Alternatively, 
instead of applying 
Theorem~\ref{positive-AFP},
there are several more direct ways of constructing an invariant measure on each of the five 
reducts of $(\Rationals, <)$.
We sketched the construction of the Glasner--Weiss measure on $(\Rationals, <)$ in 
\S\ref{background-subsection},
as the weak limit of the uniform measures on $n$-element linear orders; each of the other four also arises as the weak limit of uniform measures.

Another way to construct 
the Glasner--Weiss
measure is
as the 
ordering on the set $\Nats$ of indices induced by
an $m$-i.i.d.\ sequence of reals, where $m$ is any nondegenerate continuous probability measure on $\Reals$.  
The
invariant measures on the remaining four 
reducts may be obtained in a similar way from
an i.i.d.\ sequence on
the respective reduct of $\Reals$.

For example, for the 
countable dense circular order,
the (unique) invariant measure can be obtained
as either the weak limit of the uniform measure on circular orders of size $n$ with the (ternary) clockwise-order relation,
or from the 
ternary relation induced on the set $\Nats$ of indices by the
clockwise-ordering of an $m$-i.i.d.\ sequence, where $m$ is a nondegenerate continuous probability measure on the unit circle.

Note
that the existence of an invariant measure on the orbit of each ultrahomogeneous highly homogeneous structure $\M$ is a consequence of Exercise~5 of \cite[\S4.10]{MR1066691}; this exercise implies that the weak limit of uniform measures on $n$-element substructures of $\M$ is invariant and concentrated on the 
orbit
of $\M$.

After the following lemma, we will be able
to prove that every ultrahomogeneous highly homogeneous structure admits a unique invariant measure.
Write $S_n$ to denote the group of permutations of $\{0, 1, \ldots, n-1\}$.

\begin{lemma}
\label{at-most-one}
Let $\M \in \Str_L$.
If $\M$ is
highly homogeneous,
then there is at most
one 
invariant measure 
on the isomorphism class of $\M$ in $\Str_L$.
\end{lemma}
\begin{proof}
Let $n\in\Nats$ and let $p$ be a qf-type of $\Lwow(L)$ in $n$ variables that is realized in $\M$.
Because $\M$ is highly homogeneous, 
for any qf-type $q$ of $\Lwow(L)$ in $n$ variables that is realized in $\M$,
there is some $\tau\in S_n$ such that
\[
\M \models (\forall x_0\cdots x_{n-1}) 
\ \bigl(
p(x_0, \ldots, x_{n-1}) 
\leftrightarrow
q(x_{\tau(0)}, \ldots, x_{\tau(n-1)})\bigr) .
\]

Suppose $\mu$ is an
invariant measure on 
$\Str_L$ concentrated on the 
orbit
of
$\M$.
Then
for any $k_0, \ldots, k_{n-1} \in \Nats$, we have
\[
\mu\bigl ( \llrr{p(k_0, \ldots, k_{n-1})}  \bigr) = 
\mu\bigl ( \llrr{q(k_{\tau(0)}, \ldots, k_{\tau(n-1)})}  \bigr).
\]
By the $\sym$-invariance of $\mu$, we have
\[
\mu\bigl ( \llrr{q(k_{\tau(0)}, \ldots, k_{\tau(n-1)})} \bigr) = 
\mu\bigl ( \llrr{q(k_0, \ldots, k_{n-1})} \bigr).
\]
Let $\alpha_n$ be the number of
distinct qf-types of $\Lwow(L)$ in $n$-many variables that are realized in $\M$.
Note that $\alpha_n \leq n!$
by the high homogeneity of $\M$.
Then
\[
\mu\bigl ( \llrr{p(k_0, \ldots, k_{n-1})} \bigr) = 
\frac 1 {\alpha_n}.
\]
Sets of the form $\llrr{p(k_0, \ldots, k_{n-1})}$ generate the $\sigma$-algebra of
Borel subsets of the isomorphism class of $\M$ in $\Str_L$,
and so $\mu$ must be the unique measure determined in this way.
\end{proof}

Putting 
the previous two results 
together, we obtain the following.

\begin{proposition}
\label{combined-hh-prop}
Let $L$ be relational and let $\M\in\Str_L$ be ultrahomogeneous.
If $\M$ is
highly homogeneous, then there is a unique invariant measure on the
isomorphism class of $\M$ in $\Str_L$.
\end{proposition}
\begin{proof}
By~Lemma~\ref{highly-homog-meas-from-struct}, there is an invariant measure
on the isomorphism class of $\M$ in $\Str_L$.
On the other hand, 
by~Lemma~\ref{at-most-one}, this is the only invariant measure on the
isomorphism class of $\M$ in $\Str_L$.
\end{proof}

\subsection{Continuum-many ergodic invariant measures}

We now show that when
a countable ultrahomogeneous structure in a relational language admits an invariant measure but is not highly homogeneous,
there are continuum-many ergodic invariant measures on its orbit.
We do this by constructing a continuum-sized class of 
\emph{reweighted} measures $m^\ww$  that give rise to distinct 
measures $\mu_{(\PP, m^\ww)}$ on 
the orbit of the structure,
for some appropriate $\PP$.
This will allow us to complete the proof of our main result, Theorem~\ref{simpler-theorem}.
We start with some definitions.

\begin{definition}
A
\defn{partition} of $\Reals$
is
a collection of subsets of $\Reals$ that are non-overlapping and whose union is $\Reals$.
By \defn{half-open interval}, we mean a non-empty, left-closed, right-open interval of $\Reals$, 
including the cases
$\Reals$,
$(-\infty, c)$, and 
$[c, \infty)$ 
for $c\in\Reals$.
A \defn{weight} $\ww$ consists of a partition of $\Reals$ into a finite set
$\II_\ww$ of 
finite unions of
half-open
intervals,
along with a map $u_\ww\colon \II_\ww \to
\Rplus$ 
that assigns
a positive real number to
each element of $\II_\ww$
and satisfies
\[
\sum_{I \in \II_\ww} u_\ww(I) = 1.
\]
Given a measure $m$ on $\Reals$, the \defn{reweighting} $m^\ww$ of $m$ by a
weight $\ww$ is the measure on $\Reals$ defined by
\[m^\ww(B) = \sum_{I \in \II_\ww} u_\ww(I) \, \frac{m(B\cap I)}{m(I)}
\]
for all Borel sets $B\subseteq\Reals$.
\end{definition}

The following is immediate from the definition of a weight.

\begin{lemma}
\label{reweighting-also}
Let $m$ be a nondegenerate continuous
probability measure  on $\Reals$, and let $\ww$
be a weight. Then $m^\ww$, the reweighting of $m$ by $\ww$, is
also a nondegenerate continuous
probability measure  on $\Reals$.
\end{lemma}

We then 
obtain the following corollary.

\begin{corollary}
\label{just-as}
Let $L$ be relational, let $\PP$ be a Borel $L$-structure 
that strongly witnesses a
pithy $\Pi_2$ theory $T$, and
let $m$ be a nondegenerate continuous
probability measure  on $\Reals$. Let $\ww$
be 
a weight.
Then 
$\mu_{(\PP, m^\ww)}$
is concentrated on the set of structures in $\Str_L$ that are
models of $T$, 
just as $\mu_{(\PP, m)}$
is.
\end{corollary}
\begin{proof}
Let $L$, $\PP$, $T$, $m$, and $\ww$ be as stated. 
By Lemma~\ref{reweighting-also}, $m^\ww$ is also a nondegenerate continuous probability 
measure.
Therefore, because
$\PP$ strongly witnesses $T$, both $(\PP,m)$ and $(\PP, m^\ww)$ witness $T$.
Hence both $\mu_{(\PP, m)}$ and $\mu_{(\PP, m^\ww)}$
are concentrated
on 
the set of
models of $T$ in $\Str_L$ by Proposition~\ref{BorelLStructuresWitnessingTLeadToInvariantMeasures}.
\end{proof}

We now show that 
when $\M$ is not highly homogeneous but admits an invariant measure, reweighting can be used to obtain
continuum-many ergodic invariant measures
on the isomorphism class of $\M$ in $\Str_L$.
Specifically,
suppose $L$ is relational, $\M\in\Str_L$ is ultrahomogeneous,
and $T$ is the \Fr\ theory of $\M$.
Then as $\ww$ ranges over weights, we will see that
there are continuum-many measures 
$\mu_{(\PP, m^\ww)}$,
where $\PP$ is a Borel $L$-structure that strongly witnesses $T$
and $m$ is a nondegenerate continuous probability measure on $\Reals$. 

We start with 
two technical results.
Recall that $S_n$ is the group of permutations of $\{0, 1, \ldots, n-1\}$.

\begin{lemma}\label{equ}
Fix $n, \ell\in\mathbb{N}$.
Suppose $\{a_s\}_{s\in \{0,1,\ldots, \ell\}^n}$ is a collection of non-negative reals with the following properties:
\begin{itemize}
\item[(a)] For 
each
$\sigma\in S_n$ and $s,t\in\{0,1,\ldots, \ell\}^n$, if $s\circ \sigma=t$ then $a_s=a_t$. 
\item[(b)]
For some $s,t\in \{0,1,\ldots, \ell\}^n$, we have $a_s\neq a_t$.
\end{itemize}
Then 
as the variables $\lambda_0,\ldots,\lambda_{\ell}$ range over
positive reals
such that
\[
\label{sum}
\lambda_0+\cdots+\lambda_\ell=1,
\]
the 
polynomial
\begin{equation}\label{poly}
\sum_{s\in \{0,1,\ldots, \ell\}^n} a_s\lambda_{s(0)}\cdots\lambda_{s(n-1)}
\tag{$\spadesuit$}
\end{equation}
assumes continuum-many values.
\end{lemma}
\begin{proof}
In 
\eqref{poly}
substitute  $1-\sum_{i=0}^{\ell-1}\lambda_i$ for $\lambda_\ell$ to obtain a polynomial $P$ in $\ell$-many variables
$\lambda_0,\ldots, \lambda_{\ell-1}$.
We will show that $P$ is a non-constant polynomial, and therefore assumes continuum-many values
as
$\lambda_0,\ldots, \lambda_{\ell-1}$ 
range over positive reals such that
\[
\lambda_0+\cdots+\lambda_{\ell-1}<1.
\]
Suppose towards a contradiction that $P$ is a constant polynomial.
Let $a^*\defas a_{u}$,
where $u \in  \{0, 1, \ldots, \ell\}^n$
is the constant function taking the value $\ell$.
Consider, for
$0\le k \le n$,
the following claim $(\diamondsuit_k)$.
\begin{itemize}
\item[$(\diamondsuit_k)$] 
For every $j$ such that $0\leq j\leq k$, whenever $s\in \{0,1,\ldots, \ell\}^n$
 is such that  exactly $j$-many of $s(0),\ldots, s(n-1)$
 are different from $\ell$,  then $a_s=a^*$.
\end{itemize}
\noindent 
The statement $(\diamondsuit_n)$ implies
that for every $s\in\{0,1,\ldots, \ell\}^n$, we have $a_s=a^*$, thereby contradicting (b). Hence it suffices to prove $(\diamondsuit_n)$,
which we now do by induction on $k$.

The statement $(\diamondsuit_0)$ is clear.
Now let $k$ be such that $1\le k \le n$, and
suppose that $(\diamondsuit_{k-1})$ holds. We
will show that $(\diamondsuit_k)$ holds.
Let
$s \in \{0,1,\ldots, \ell\}^n$  be such that  exactly $k$-many of $s(0),\ldots, s(n-1)$
 are different from $\ell$; we 
must prove
that  $a_s=a^*$.

Since $(\diamondsuit_{k-1})$ holds, by
(a)
we may assume
without loss of generality
that none of $s(0), s(1), \ldots, s(k-1)$ equals $\ell$ and that
\[s(k)=s(k+1)=\cdots=s(n-1)=\ell.\]
For $0 \le r \le \ell-1$ let $k_r$ denote the number of times that $r$ appears in the sequence $s(0), \ldots, s(k-1)$. In particular,
 \[\lambda_{s(0)}\lambda_{s(1)}\cdots \lambda_{s(k-1)}=\lambda_0^{k_0}\lambda_1^{k_1}\cdots\lambda_{\ell-1}^{k_{\ell-1}},\]
and $k = k_0+k_1+\cdots +k_{\ell-1}$.

Let $\beta$ be the coefficient of $\lambda_0^{k_0}\lambda_1^{k_1}\cdots\lambda_{\ell-1}^{k_{\ell-1}}$ in $P$.
 For $t_0, \ldots, t_{\ell-1}\in\Nats$
such that $t_0+\cdots +t_{\ell-1}\leq n$,
 let $\Gamma(t_0,t_1,\ldots,t_{\ell-1}) \in \{0, 1, \ldots, \ell\}^n$ 
be the
non-decreasing sequence of length $n$
consisting of $t_0$-many $0$'s, $t_1$-many $1$'s, \ldots, $t_{\ell-1}$-many  $t_{\ell-1}$'s, and $(n- \sum_{i=0}^{\ell-1} t_i)$-many $\ell$'s.
Define $C\defas \frac{n!}{k_0! k_1! \cdots k_{\ell-1}! (n-k)!}$. Then
\begin{equation}\label{heart}
\beta= 
C \sum_{t_i\leq k_i} a_{\Gamma(t_0,t_1,\ldots,t_{\ell-1})} {k_0 \choose t_0} \cdots {k_{\ell-1} \choose t_{\ell-1}}  (-1)^{k-\sum_{i=0}^{\ell-1} t_i}.
\tag{$\varheart$}
\end{equation}

Note that $a_s = a_{\Gamma(k_0,k_1,\ldots,k_{\ell-1})}$.
By $(\diamondsuit_{k-1})$, we also have $a^* = a_{\Gamma(t_0,t_1,\ldots,t_{\ell-1})}$
if $\sum_{i=0}^{\ell-1} t_i < k$, in particular whenever
each $t_i \le k_i$ for $0\le i \le \ell-1$ and 
$(t_0, t_1, \ldots, t_{\ell-1}) \neq (k_0,k_1,\ldots,k_{\ell-1})$. In other words, 
all 
subexpressions \linebreak
$a_{\Gamma(t_0,t_1,\ldots,t_{\ell-1})}$ appearing in \eqref{heart}
other than (possibly) $a_{\Gamma(k_0,k_1,\ldots,k_{\ell-1})}$ are
equal to $a^*$.

By the multinomial and binomial theorems,
\begin{eqnarray*}
\sum_{t_i\leq k_i} {k_0 \choose t_0} \cdots {k_{\ell-1} \choose t_{\ell-1}}  
(-1)^{k-\sum_{i=0}^{\ell-1} t_i}
&=& 
\sum_{t\leq k}\sum_{\substack{t_i\leq k_i \\ \sum_i t_i=t}} {k_0 \choose t_0} \cdots{k_{\ell-1} \choose t_{\ell-1}}  (-1)^{k-\sum_{i=0}^{\ell-1} t_i}\\
&=& \sum_{t\leq k}{k \choose t}   (-1)^{k-t} 
= (1-1)^k 
= 0.
\end{eqnarray*}
Therefore
\[
\beta= 
0 + (a_s - a^*) {k_0 \choose k_0} \cdots {k_{\ell-1} \choose k_{\ell-1}}
(-1)^{k-k}
= a_s - a^*.
\]

But by the assumption that  $P$ is a constant polynomial, 
$\beta= 0$, and so $a_s = a^*$, as desired.
\end{proof}

Using this lemma, we can prove the following.

\begin{proposition}\label{cmeas}
Let $m$ be a nondegenerate continuous probability measure, and let 
$n$ be a positive integer.
Suppose
$A\subseteq\Reals^n$ 
is
an $S_n$-invariant Borel
set such that $0< m^n(A)<1$. Then the family of reals $\{(m^\ww)^n(A)\colon \ww \text{  is a weight}\}$ has cardinality equal to the continuum.
\end{proposition}
\begin{proof}
Because
$m^n(A)>0$,
we may define $\tm$, the conditional distribution of $m^n$ given $A$,
by \[\tm(B)\defas\frac{m^n(A\cap B)}{m^n(A)}\]
for every Borel set $B\subseteq \Reals^n$.

Because $m^n(A)<1$, we have
$m^n(\Reals^n - A)=1-m^n(A)>0$.
Furthermore
$\tm(\Reals^n - A)=0$,
and so
$m^n\neq \tm$. 
Therefore there 
are some 
half-open intervals $X_0, X_1, \ldots, X_{n-1}$
such that
\[\textstyle
m^n\bigl(\prod_{i=0}^{n-1} X_i\bigr)\neq \tm\bigl(\prod_{i=0}^{n-1} X_i\bigr).
\]
Because $m$ is nondegenerate, $m(X_i)>0$ for each $i\le n-1$.

Define  
the 
partition 
$\JJ$
of $\Reals$ 
to be the family of 
non-empty
sets 
of the form
$X_0^{e_0}\cap\ldots\cap X_{n-1}^{e_{n-1}}$
for some 
$e_0,\ldots,e_{n-1}\in\{+, -\}$,
where $X_j^+\defas X_j$ and
$X_j^-\defas \Reals - X_j$.
Let $\ell \defas |\JJ| - 1$, and let $Y_0, \ldots, Y_\ell$ be some enumeration of $\JJ$.

For $s\in \{0,1,\ldots, \ell\}^n$, set 
\[
a_s\defas 
\frac{m^n\bigl(A\cap \prod_{i=0}^{n-1} Y_{s(i)}\bigr)}{m^n\bigl(\prod_{i=0}^{n-1} Y_{s(i)}\bigr)}.
\]
Note that there exists 
some
$\ts \in \{0,1,\ldots, \ell\}^n$
such that
\[\textstyle
m^n\bigl(\prod_{i=0}^{n-1} Y_{\ts(i)}\bigr)
\neq
\tm \bigl(\prod_{i=0}^{n-1} Y_{\ts(i)}\bigr),
\]
i.e.,
\[\textstyle
m^n\bigl(\prod_{i=0}^{n-1} Y_{\ts(i)}\bigr)
\neq 
\dfrac{m^n\bigl(A\cap \prod_{i=0}^{n-1} Y_{\ts(i)}\bigr)}{m^n(A)}
,
\]
and hence
\[
m^n(A)
\neq 
\frac{m^n\bigl(A\cap \prod_{i=0}^{n-1} Y_{\ts(i)}\bigr)}{m^n\bigl(\prod_{i=0}^{n-1} Y_{\ts(i)}\bigr)}
= a_v
.
\]

Observe that
if 
for $s\in \{0,1,\ldots, \ell\}^n$, 
the values $a_s$
are all equal, then 
this value is
$m^n(A)$.  But we have just shown that 
$a_\ts \neq m^n(A)$,
 and so
$a_s\neq a_t$
for some $s,t$.
Further, since
$A$ is $S_n$-invariant, from the definition of 
$a_s$
we have that
for every $\sigma\in S_n$, and every $s$ and $t$, if $s\circ\sigma=t$ then $a_s=a_t$.

Hence the assumptions of Lemma~\ref{equ} are satisfied, and so
the expression
\[ \sum_{s\in \{0,1,\ldots, \ell\}^n} a_s\lambda_{s(0)}\cdots\lambda_{s(n-1)}\]
takes
continuum-many 
values as
$\lambda_0,\ldots,\lambda_\ell$ range over positive reals
satisfying $\lambda_0+\cdots+\lambda_\ell=1$.
Each 
such
$\lambda_0,\ldots,\lambda_\ell$ together with the partition $\JJ$ 
yields
a weight $\ww$ via $u_\ww(Y_i) \defas \lambda_i$ for $i \le \ell$.
Then
the corresponding reweightings $m^\ww$ satisfy
\begin{eqnarray*}
(m^\ww)^n(A) &=&
\sum_{s\in \{0,1,\ldots, \ell\}^n} 
\frac{m^n\bigl(A\cap \prod_{i=0}^{n-1} Y_{s(i)}\bigr)}{m^n\bigl(\prod_{i=0}^{n-1} Y_{s(i)}\bigr)}
\lambda_{s(0)}\cdots\lambda_{s(n-1)}\\
&=&
\sum_{s\in \{0,1,\ldots, \ell\}^n} a_s\lambda_{s(0)}\cdots\lambda_{s(n-1)}.
\end{eqnarray*}
We conclude that the family $\{(m^\ww)^n(A)\colon \ww \text{  is a weight}\}$ has cardinality equal to the continuum.
\end{proof}

Now we may prove our main result about weights.

\begin{proposition}
\label{key-prop}
Let $L$ be relational, let $\M\in\Str_L$ be ultrahomogeneous, and let
$T$ 
be the \Fr\ theory of $\M$.
Suppose that $\M$ is not highly homogeneous.
Further suppose
that 
$\PP$ 
is 
a Borel $L$-structure 
that strongly
witnesses $T$, and that
$m$ is a nondegenerate continuous probability measure on $\Reals$.
Then 
there are continuum-many measures
$\mu_{(\PP, m^\ww)}$, 
as $\ww$ ranges over weights.
\end{proposition}
\begin{proof}
By Proposition~\ref{BorelLStructuresWitnessingTLeadToInvariantMeasures},
$\mu_{(\PP, m)}$ is concentrated on the orbit of $\M$. 
Because $\M$ is not highly homogeneous, by Lemma~\ref{ultra-highhom}
there are non-isomor\-phic $n$-element substructures
$A_0$, $A_1$ 
of $\M$ 
for some $n\in\Nats$.
Fix an enumeration of each of $A_0$, $A_1$, and 
for $i = 0, 1$
let
$\varphi_i$ be a quantifier-free $\Lwow(L)$-formula in $n$-many free variables 
that 
is satisfied by
$A_i$ 
and 
not by
$A_{1-i}$
(in their respective enumerations).
Note that
\[
0< \mu_{(\PP, m)}\Bigl(
\bigllrr{\bigvee_{\sigma\in S_n}\varphi_i(\sigma(0), \ldots, \sigma(n-1))}
\Bigr)
\]
for $i=0,1$, as $\varphi_i$ is realized in $\M$. Furthermore, as
\[
\mu_{(\PP, m)}\Bigl(\bigllrr{\bigvee_{\sigma\in S_n}\varphi_0(\sigma(0), \ldots, \sigma(n-1))} \Bigr) 
\,
+
\,
\mu_{(\PP, m)}\Bigl(\bigllrr{\bigvee_{\sigma\in S_n}\varphi_1(\sigma(0), \ldots, \sigma(n-1))} \Bigr) 
\le 1
,
\]
we have
\[
\mu_{(\PP, m)}\Bigl(\bigllrr{\bigvee_{\sigma\in S_n}\varphi_i(\sigma(0), \ldots, \sigma(n-1))} \Bigr)
< 1
\]
for $i=0,1$.

Then, by Lemma~\ref{mu-vs-m-n}, we have
\[
0 < m^n \Bigl(\bigl\{ (a_0,\ldots, a_{n-1})  \in \Reals^n \st \PP \models \bigvee_{\sigma\in S_n}\varphi_0 (a_{\sigma(0)}, \ldots, a_{\sigma(n-1)})\bigr \}\Bigr)< 1.
\]

Hence by
Proposition~\ref{cmeas},
as $\ww$ ranges over weights,
\[
(m^\ww)^n \Bigl(\bigl\{ (a_0,\ldots, a_{n-1})  \in \Reals^n \st \PP \models \bigvee_{\sigma\in S_n}\varphi_0 (a_{\sigma(0)}, \ldots, a_{\sigma(n-1)})\bigr \}\Bigr)
\]
takes on continuum-many values.
Again by 
Lemma~\ref{mu-vs-m-n}, 
\[
\mu_{(\PP, m^\ww)}\Bigl(\bigllrr{\bigvee_{\sigma\in S_n}\varphi_i(\sigma(0), \ldots, \sigma(n-1))} \Bigr) 
\]
takes on continuum-many values as $\ww$ ranges over weights;
in particular, 
the
$\mu_{(\PP, m^\ww)}$
constitute continuum-many different measures.
\end{proof}

We are now able to complete the proof of our main theorem.

\proofof{Theorem~\ref{simpler-theorem}}
Given a countable structure $\N$ in a countable language,
by
Corollaries~\ref{wlog-corollary}
and \ref{trivialdcl-canonical}
and Lemma~\ref{wlog-lemma}, 
its canonical structure $\Nbar$ admits the same number of ergodic invariant measures as $\N$, is highly homogeneous if and only if $\N$ is, and has trivial definable closure if and only if $\N$ does. Hence it suffices to prove the theorem in the case where $\M\in\Str_L$
is the canonical structure of some countable structure in a countable language; in particular, 
where $\M$ is ultrahomogeneous 
(by Proposition~\ref{canonical-is-ultrahom})
and $L$ is relational (by Definition~\ref{def-canonical}).

By Theorem~\ref{negative-AFP},
if $\M$ has nontrivial
definable closure then its 
orbit
does not admit an invariant
measure, as claimed in (0). 

By Proposition~\ref{combined-hh-prop}, 
if
$\M$ is highly homogeneous
then 
its
orbit
admits a
unique invariant measure, as claimed in (1).

Clearly, the 
orbit
of $\M$
admitting 
0, 1, or continuum-many invariant measures are
mutually exclusive possibilities.
Hence
it remains to show that if $\M$ is not highly homogeneous and 
its 
orbit
admits an invariant measure, then 
this 
orbit
admits continuum-many ergodic invariant measures.

Again by Theorem~\ref{negative-AFP},
because the 
orbit
of $\M$ admits an
invariant measure, $\M$ must have trivial definable closure. 
Since $\M$ is ultrahomogeneous and $L$ is relational, by Proposition~\ref{AFP-Borel-L}
there is 
a Borel $L$-structure $\PP$ 
that strongly
witnesses the \Fr\ theory of $\M$.

Let $m$ be a nondegenerate continuous probability measure on $\Reals$.
By Proposition~\ref{key-prop},
as $\ww$ ranges over weights, there are continuum-many different measures
$\mu_{(\PP, m^\ww)}$.
By
Corollary~\ref{just-as},
each is an
invariant probability measure
concentrated on the 
orbit
of $\M$, and
by Proposition~\ref{ergodic-measure},
each 
is ergodic. Finally, there are at most continuum-many Borel measures on $\Str_L$.
\Endproofof




\vspace*{20pt}
\section*{Acknowledgments}

The authors would like to thank
Willem Fouch{\'e}, Yonatan Gutman, Alexander Kechris,
Andr{\'e} Nies, Arno Pauly, Jan Reimann, and Carol Wood
for helpful conversations.

This research was facilitated by the Dagstuhl Seminar on Computability,
Complexity, and Randomness (January 2012), the
conference on Graphs and
Analysis  at the
Institute for Advanced Study (June 2012),
the Buenos Aires Semester in Computability, Complexity, and Randomness
(March 2013),
the Arbeitsgemeinschaft on Limits of Structures at the Mathematisches
Forschungsinstitut Oberwolfach (March--April 2013), the Trimester
Program on Universality and Homogeneity of the Hausdorff
Research Institute for Mathematics at the University of Bonn
(September--December 2013), and the workshop on Analysis, Randomness, and
Applications at the University of South Africa (February 2014).

Work on
this publication by C.\,F.\ was made possible through the support of
ARO grant W911NF-13-1-0212
and grants from the
John Templeton Foundation and Google.
The opinions expressed in this
publication are those of the authors and do not necessarily reflect the
views of the John
Templeton Foundation.

\vspace*{30pt}

\bibliographystyle{amsnomr}

\providecommand{\bysame}{\leavevmode\hbox to3em{\hrulefill}\thinspace}
\providecommand{\MR}{\relax\ifhmode\unskip\space\fi MR }
\providecommand{\MRhref}[2]{%
  \href{http://www.ams.org/mathscinet-getitem?mr=#1}{#2}
}
\providecommand{\href}[2]{#2}

\end{document}

%% file: macros.tex
\theoremstyle{plain}
\newtheorem{proposition}{Proposition}[section]
\newtheorem{lemma}[proposition]{Lemma}
\newtheorem{corollary}[proposition]{Corollary}

\newtheorem{theorem}[proposition]{Theorem}

\theoremstyle{definition}
\newtheorem{definition}[proposition]{Definition}

\theoremstyle{remark}

\newcommand{\proofof}[1]{\ \\\noindent \emph{Proof of {#1}.}}
\newcommand{\Endproofof}{\hfill\qed\ \\\ \\}

\DeclareMathOperator*{\Aut}{Aut}

\DeclareMathOperator{\dcl}{dcl}

\newcommand{\llrr}[1]{{\llbracket #1 \rrbracket}}
\newcommand{\bigllrr}[1]{{\bigl \llbracket #1 \bigr \rrbracket}}

\newcommand{\Fr}{Fra\"{i}ss\'{e}}

\newcommand{\FP}{{\mathcal{F}_{\mathcal{P}}}}

\newcommand{\tm}{\widetilde{m}}
\newcommand{\ts}{{v}}

\newcommand{\Mbar}{\overline{\M}}
\newcommand{\Nbar}{\overline{\N}}
\newcommand{\Gbar}{{\mathcal{C}_G}}

\DeclareSymbolFont{extraup}{U}{zavm}{m}{n}
\DeclareMathSymbol{\varheart}{\mathalpha}{extraup}{86}
\DeclareMathSymbol{\vardiamond}{\mathalpha}{extraup}{87}

\def\a{{\mathbf{a}}}

\def\t{{\mathbf{t}}}
\def\x{{\mathbf{x}}}
\def\y{{\mathbf{y}}}

\def\M{{\EM{\mathcal{M}}}}
\def\N{{\EM{\mathcal{N}}}}
\def\JJ{{\EM{\mathcal{J}}}}

\def\PP{{\EM{\mathcal{P}}}}
\def\II{{\EM{\mathcal{I}}}}

\def\ww{{\EM{\mathcal{W}}}}

\newcommand{\sympar}[1]{\ensuremath{S_{#1}}}
\newcommand{\sym}{\sympar{\infty}}

\def\Str{{\mathrm{Str}}}

\def\Lomega#1{{\EM{\mc{L}_{#1, \w}}}}

\def\Lwow{\Lomega{\w_1}}

\newcommand{\st}{{\,:\,}}
\newcommand{\Nats}{\Naturals}

\def\w{\EM{\omega}}
\def\Naturals{{\EM{{\mbb{N}}}}}

\def\Rationals{{\EM{{\mbb{Q}}}}}
\def\Reals{{\EM{{\mbb{R}}}}}
\def\Rplus{{\EM{{\mbb{R}^+}}}}

\def\^{\EM{{}^{\And}}}

\def\And{\EM{\wedge}}

\def\<{\EM{\langle}}
\def\>{\EM{\rangle}}

\def\indent{\hspace*{2em}}

\newcommand{\defn}[1]{\textbf{#1}}
\newcommand{\defas}{:=}

\def\EM#1{\ensuremath{#1}}

\def\mbb#1{\EM{\mathbb{#1}}}

\def\mc#1{\EM{\mathcal{#1}}}

\newcommand{\ER}{Erd\H{o}s--R\'{e}nyi}

\definecolor{MyGreen}{rgb}{.75,0,.75}
\definecolor{RealGreen}{rgb}{0,1,0}
\definecolor{DarkGreen}{rgb}{0.1,0.4,0.1}
\definecolor{ActualGreen}{rgb}{0.0,0.5,0.0}
\definecolor{MyBlue}{rgb}{0,0,1}
\definecolor{MyRed}{rgb}{1,0,0}
\definecolor{DarkRed}{rgb}{0.5,0,0}
\definecolor{Purple}{rgb}{0.4,0,0.4}
\definecolor{Orange}{rgb}{0.85,0.45,0.45}

\definecolor{WowColor}{rgb}{.75,0,.75}
\definecolor{SubtleColor}{rgb}{0,0,.50}

\newcounter{margincounter}

\newcommand{\HideNate}[1]{}